\documentclass{amsart}
\usepackage{ericstyle}
\title[Weak amenability is stable under graph products]{Weak Amenability is Stable Under Graph Products}
\date{\today}
\author{Eric Reckwerdt}

\begin{document}
\maketitle
\begin{abstract}
	Weak amenability of discrete groups was introduced by Haagerup and co-authors in the 1980's. 
	In particular, weak amenability in known to be stable under taking direct products and free products.
	In this paper we show that weak amenability (with Cowling-Haagerup constant 1) is stable under taking graph products of discrete groups.
	Along the way we will construct a wall space  associated to the word length structure of a graph product and also give a method of extending completely bounded functions on discrete groups to a completely bounded function on their graph product.
\end{abstract}

\section{Introduction}
Amenability has been widely studied since the early twentieth century, finding applications across many fields of mathematics. 
The idea was first studied by von Neumann and the term amenable was later coined by Day.
An amenable group is one that has an invariant mean; in other words there is a way to uniformly average a bounded function over the entire (infinite) group.
This can also be seen as the ability to approximate the group by finite subsets, as is shown by the existence of F\o lner sequences:
if a group is amenable then one can find a sequence of finite subsets such that the limit of averages over the sequence is a uniform average over the whole group.
Interest in other finite approximation properties like this have lead to a variety of generalizations of amenability, including a-T-menability (introduced by Gromov \cite{Gro88}, see \cite{ChCoVal01} for a useful reference) and weak amenability (studied by Haagerup and various co-authors \cite{DeCaHa85}, \cite{CowHa89}, \cite{Haag78}).

Weak amenability and a-T-menability are closely linked, as can be seen in the following characterization.
An approximate identity of a group $G$ is a sequence of functions $\{\f_n\}$  from $G$ to $\C$ such that $\f_n \to \mathbf{1}$ pointwise as $n$ goes to infinity.
A group is amenable if there is an approximate identity consisting of finitely supported positive definite functions,
a-T-menable if there is an approximate identity consisting of positive definite functions which vanish at infinity,
and is weakly amenable   if there is an approximate identity consisting of finitely supported functions which are uniformly completely bounded with cb-norm approaching one (we always assume with Cowling-Haagerup constant one).
Note that a positive definite function $\f$ is completely bounded  with cb-norm $\cb{\f}= \f(e)$, and if a group is amenable then it is both a-T-menable and weakly amenable.

Furthermore, there are many examples of groups which are non-amenable but both  weakly amenable and a-T-menable, such as 
free products of amenable groups with finite amalgam, Coxeter groups, and groups acting on CAT(0)-cube complexes.
These similarities led to the conjecture that the two properties were equivalent,
but it was shown that a-T-menability is stable under wreath products (De Cornulier, Stalder, and Valette \cite{dCSV12}) but weak amenability is not (Ozawa and Popa \cite{OzPo10}).
In this paper we study weak amenability.

It was first shown by deCanierre and Haagerup \cite{DeCaHa85} that the free group on $2$ generators, the classic non-amenable group, is weakly amenable.
This came from their study on the weak amenability of rank one Lie groups.  
Later, Bozejko and Picardello \cite{BoPi93} showed free products (with finite amalgam) of amenable groups are weakly amenable, following the treatment of Haagerup in his seminal paper \cite{Haag78} on the C$^*$-algebra of the free group.
In particular, they presented a cb-norm bound on the characteristic function of the n-sphere for a tree.
The stability of weak amenability under free products was finally proved by Ricard and Xu \cite{RiXu06} in 2006.
In 2008, Mizuta extended Bozejko and Picardello's technique to groups acting properly on finite dimensional CAT(0)-cube complexes, reproving a result of  Guentner and Higson \cite{GuHi10}.

A generalization of free products is the graph product:
given a simplicial graph $\Gamma$ and a set of discrete groups $\{G_v\}$ indexed by the vertices of $\Gamma$, there is the associated graph product $G(\Gamma)$, which is the free product of the $G_v$'s with the relation that for each edge $[vw]$ of $ \Gamma$,  elements of $G_v$ and $G_w$ commute with each other.
The word length structure of a graph product is a CAT(0)-cube complex, as we demonstrate in this paper. 
 In 2013, Antolin and Dreesen \cite{AnDr13} showed that a-T-menablity is stable under graph products.
	This, combined with Mizuta's and Ricard and Xu's results, strongly indicated that weak amenability should be stable under graph products as well.

In this paper we prove the following:
\begin{theorem*}
	The graph product over a finite graph of weakly amenable groups \textup{(}with Cowling-Haagerup constant 1\textup{)}  is weakly amenable.
\end{theorem*}
We begin with background on weak amenability and graph products.
Section 3 describes a wall space on which a given graph product acts.
From this, using the correspondence between CAT(0) cube complexes and wall spaces, we obtain a CAT(0) cube complex that describes the word length structure of the graph product.
In Section 4, we treat the extension of positive definite functions on groups to a positive definite function on their free product.
For this we use an infinite tensor product technique, similar to the one used by Chen, Dadarlat,  Guentner, and Yu \cite{GuDaChYu03} to show that free products with finite amalgam preserve uniform embeddability.
We then modify this construction to extend completely bounded functions to a free product.
This motivates the proof of the main theorem in Section 5, where we generalize the extension to graph products and construct a completely bounded approximate identity from such extensions.
These results contribute to my doctoral thesis at the University of Hawai\kern.05em`\kern.05em\relax i at \Manoa.

Finally, I would like to thank the referee for their comments and my advisor Erik Guentner for all of his help and support in my pursuit of mathematics.

\section{Background}
	\subsection{Operator norms and weak amenability}

	Completely bounded maps and the completely bounded norm come from a collection of natural matrix norms associated to any C$^*$-algebra.
	If $C$ is a C*-algebra, then $M_n(C)$, the $n\times n$ matrices with entries in $C$, is also a $C^*$-algebra which gives it an unique norm.
	Any (bounded, linear) operator $\f$ on $C$ extends to an operator $\f^n$ on $M_n(C)$, acting entrywise as $\f$. 
	An operator $\f$ is \textit{completely bounded} (\textit{cb}) if all of the $\f^n$'s are bounded uniformly.
	As the norms are increasing with $n$, we can define the \textit{cb-norm} of $\f$: $\cb{\f} := \lim_n\|\f^n\|$.
	The operator $\f$ is \textit{completely positive} if $\f^n$ is positive (as an operator on $M_n(C)$) for all $n$. 
	A completely positive operator $\f$ is completely bounded with $\cb{\f} = \f(\mathbf{1})$.
	For more details, see Paulsen's book \cite{Paul02}.

	Given a group $G$ and a map $\f: G\to \C$, we can naturally extend $\f$ to a multiplication operator on the group ring $\C G$, with $\sum a_g g \mapsto a_gf(g) g$. 
	If $\f$ is finitely supported, as is the case for functions in this paper, $\f$ can be extended to a multiplication operator on $C^\ast_r(G)$, the reduced C$^*$-algebra of $G$.	
	If $\f$ is completely bounded as a multiplication operator, then we say that $\f$, as a function, is completely bounded.

	Recall that a function $\f$ is \textit{positive definite} if and only if for all finite sets of complex numbers $\{c_i\}$ and group elements $\{g_i\}$, $\sum_{i,j}c_i\bar c_j \f(g_j^{-1}g_j)\ge 0$.
	This implies that, as a multiplication operator on $C^*_r(G)$, $\f$ is competely positive.
	Hence, a positive definite function $\f$ is completely bounded with $\cb{\f} = \f(\mathbf{e})$.  

	For most of the paper we shall use an alternate characterization of completely bounded functions, (see Theorem 8.3 and Remark 8.4 in \cite{Pisi03}):
	\begin{proposition}\label{GNS}
		If a function $\f: G\to \C$ is completely bounded then there is a Hilbert space $\scr{H}$ and maps $\alpha,\beta: G \to \scr{H}$ with $\|\alpha\|\|\beta\| = \cb{\f}$, such that $\f(h^{-1}g) = \langle \alpha(g),\beta(h)\rangle$ for all $g,h\in G$.
		Conversely, given a function of the form $\f(h^{-1}g) = \langle \alpha(g),\beta(h)\rangle$, then $\f$ is completely bounded with $\cb{\f}\le \|\alpha\|\|\beta\|$.

		If $\f$ is positive definite, then we may take $\alpha=\beta$ and $\|\alpha\|^2 = \f(e)$.
		Conversely, given a function of the form $\f(h^{-1}g) = \langle \alpha(g),\alpha(h)\rangle$, then $\f$ is positive definite.
Note that this paragraph is just the $GNS$-construction.
	\end{proposition}

	Recall that a \textit{kernel} on a group is a map $\phi:G\times G \to \C$ and is \textit{$G$-invariant} if $\phi(kg,kh)=\phi(g,h)$ for any $g,h,k \in G$.
	For any function $\f:G\to \C$, there is an associated $G$-invariant kernel $\phi:G\times  G\to \C$, where $\phi(g,h) :=\f(h^{-1}g)$.
	The above proposition characterizes completely boundedness and positive definiteness in terms of the structure of the kernel of a group function.
	In this paper we often construct a kernel in the above form, then prove that it is $G$-invariant to show that the function $\f(h^{-1}g):=\phi(g,h)$ is well defined.

	A group $G$ is \textit{amenable} if there exists a net of finitely supported positive definite functions $\{\f_n\}$ such that $\f_n\to \mathbf{1}$ pointwise.
	A group $G$ is \textit{weakly amenable} if there exists a net of finitely supported completely bounded functions $\{\f_n\}$ such that $\f_n\to \mathbf{1}$ pointwise and $\limsup \cb{\f_n} \le c < \infty$.  
	The \textit{Cowling-Haagerup constant} $\Lambda_{cb}(G)$ is the smallest such $c$, and in this paper we shall always assume $c=1$.
	Note, the main construction relies on the fact that a completely bounded function of norm close to one is not very different from a positive definite function, and may not work with Cowling-Haagerup constant bounded away from one.

	\subsection{Tensor products of Hilbert spaces}	\label{hilbert}
	For most of sections 4 and 5 we use a tensor product of Hilbert spaces over an infinite index set.
	Let us recall how to construct such an object.
	Suppose $T$ is an infinite set and for every $t\in T$, we have a Hilbert space  $\scr{H}_t$ and an unit vector $\omega_t$ of $\scr{H}_t$ called the \textit{vacuum vector} of $H_t$.
	Given a finite subset $S \subset T$, we can construct the standard tensor product over $S$, denoted $\bigotimes_{t\in S} \scr{H}_t$.  
	We write an elementary tensor $\xi$ as $\xi_{t_1}\otimes \xi_{t_2} \otimes \cdots \otimes \xi_{t_n}$, where $t_i \in S$ and $\xi_{t_i}$ is the $t_i$ component of $\xi$.
	For another finite subset $R$ of $T$ such that $R\subset S$ there is an inclusion of $\bigotimes_{t\in R} \scr{H}_t$ into $\bigotimes_{t\in S}\scr{H}_t$ with $\bigotimes_{t\in R} \xi_t \mapsto \bigotimes_{t\in R} \xi_t \otimes \bigotimes_{t\in S\setminus R} \omega_t$.
	The inner product of any two finite tensors  $\langle \bigotimes_{t\in S}\xi_t,\bigotimes_{t\in S'} \xi_s\rangle$ can then be evaluted in $\bigotimes_{t\in S\cup S'}\scr{H}_t$.

	Define $ \scr{\widehat{H}} :=
	\bigotimes_{t \in T} \scr{H}_t$ 
	as the direct limit of tensor products over finite subsets of $T$.  
	The span of finite tensors is dense in $\scr{\widehat H}$ and in this paper we shall only be concerned with elementary finite tensors.
	It is useful to think of such tensors as infinite tensors with only finitely many components not equal to $\omega_t$.

	\subsection{Graphs and Graph Products}

	Graph products were first studied by E. Green in her thesis \cite{Green90} as a generalization of several familiar classes of groups including graph groups, free products, direct products, and right-angled Coxeter groups. 
	In this section we establish basic notation for graph products and recall several needed results.

	A \textit{graph} in this paper will be a \textit{finite simplicial graph}. 
	In other words a graph $\Gamma$ is a finite set of vertices and undirected edges between those vertices, without loops or multiple edges.
	A typical vertex of $\Gamma$ will be labeled as $v\in \Gamma$, and 
	the (unique) edge between vertices $v$ and $w$ of $\Gamma$ will be denoted $[vw]\in\Gamma$.

	Two vertices are \textit{adjacent} if they share an edge.
	Since $\Gamma$ is finite, there is a maximal number of mutually adjacent vertices, which we shall denote by $\kappa$ throughout this paper.
	The \textit{link} of a vertex $v$  is the set $\lk(v)$ of vertices which are adjacent to $v$ in $\Gamma$.
	The \textit{star} of a vertex $v$ is the set $\st(v)$ of vertices adjacent to $v$ and $v$ itself.

	Given a graph $\Gamma$ and set of groups $\{G_v\}$ labeling the vertices $v\in \Gamma$, called \textit{vertex groups}, we may define their graph product.
	Informally, their graph product is the free product of the vertex groups, where two vertex groups commute if their vertices share an edge in $\Gamma$.
	Formally, we have the following:
	the \textit{graph product of $\{G_v\}_{v\in \Gamma}$ over $\Gamma$} is
	$$G(\Gamma) := \ast_{v\in \Gamma} G_v/\langle\langle [G_v,G_w], [vw] \in \Gamma\rangle\rangle.$$
	Here, as usual, $[G_v,G_w] = \langle g_v^{-1}g_w^{-1}g_vg_w|$ $g_v\in G_v,$ $g_w\in G_w\rangle$ is the commutator of $G_v$ and $G_w$ and $\langle\langle H\rangle\rangle$ denotes the normal closure of a subgroup $H$ in $G$.

	In particular, a graph product over a totally disconnected graph (one with no edges) is a free product of the vertex groups.  A graph product over a complete graph (one in which every pair of vertices are adjacent) is a direct product of the vertex groups.
	Graph products thus act as combinatorial interpolations between the direct product and the free product of a set of groups.
	Further, a graph product is a right-angled Artin group if all of the vertex groups are infinite cyclic, and is a right-angled Coxeter group when the vertex groups are isomorphic to $\mathbb{Z}_2$.

	In a graph product $G(\Gamma)$, we denote by $G(\st(v))$ the subgroup generated by $\bigcup_{w\in \st(v)} G_w$.
	Likewise $G(\lk(v))$ is the subgroup generated by $\bigcup_{w\in \lk(v)} G_w$.
	Note that elements of $G(\lk(v))$ commute with elements of $G_v$.
	We denote by $G_v^e := G_v\setminus \{e\}$, the set of non-identity elements of $G_v$.

	Given a graph product $G(\Gamma)$, a \textit{formal word} in the vertex groups is a sequence of the form $(g_1,g_2,\dots, g_n)$, where $g_i \in G(i)$ (we write $G(i)$ for $G_{v_i}$ to avoid cluttered notation).
	We call each $g_i$ a \textit{syllable} and define the \textit{syllable length} of a formal word as the number of syllables in the word.

	Each formal word represents some element of $G(\Gamma)$ via multiplication. 
	For instance, the formal word $(g_1,g_2,\dots, g_n)$ represents $g_1g_2\dots g_n$. 
	There are three operations we can perform on this word that will not change the group element it represents.
	First, if $v_i$ and $v_{i+1}$ are adjacent vertices, then $g_i$ and $g_{i+1}$ commute and we may perform a \textit{syllable shuffle}: sending $(g_1,\dots, g_i,g_{i+1},\dots, g_n)$ to $(g_1,\dots, g_{i+1},g_i, \dots, g_n)$.
	If two adjacent syllables of a word are in the same vertex group, then as group elements the two syllables will multiply, and we may amalgamate the two syllables into a single syllable, reducing the length of the word by $1$.
	Last, if one of the syllables is the identity, we may remove it, also reducing the length of the word by $1$.
	A \textit{reduced word} is a word that cannot be made shorter through syllable shufflings and cancellation/amalgamation.  

	If $g \in G(\Gamma)$, we may write it as a product: $g=h_1h_2\dots h_m$, $h_i \in G(i)$.  
	We can then reduce the formal word $(h_1,\dots, h_m)$ to some reduced word $(g_1,\dots, g_n)$ such that $g = g_1g_2\dots g_n$.
	Green showed in her thesis that this process generates an unique family of reduced words representing $g$, and that any two reduced words representing $g$ differ from each other by a finite number of syllable shuffles.
	This implies  that the set of syllables representing $g$ is unique.
	We write $g\equiv g_1g_2\dots g_n$ when $(g_1,g_2,\dots,g_n)$ is a reduced word representing $g$ 
	and we define the \textit{reduced word length} $|g|_r$ as the syllable length of any (and all) of the reduced words representing $g$.

	\begin{remark}\label{remk}
		An important observation is that if $(g_1,\dots, g_n)$ is a reduced word with two syllables $g_i$ and $g_j$ from the same vertex group $G_v$, then there must be a $g_q$ with $i<q<j$ such that $g_q \not\in G(\st(v))$ (Equivalently, $v_q$ is not adjacent to $v$).
	Otherwise, as $g_i$ and $g_j$ commute with anything in $G(\st(v))$ by definition, we could shuffle $g_i$ next to $g_j$ and reduce the word further.
\end{remark}

Given two elements of $G(\Gamma)$ we will often need to know the explicit form of a reduced word representing their product. 
Recall that $\kappa$ is the maximal number of mutually adjacent vertices in $\Gamma$.

\begin{lemma}\label{hg}
	Suppose $g$ and $h$ are elements of $G(\Gamma)$, where $m=|h|_r$ and $n=|g|_r$.
	Then there exists natural numbers $q\le \min(m,n)$, $p \le \kappa$ and reduced words $g_1\dots g_n$ and $h_1\dots h_m$ representing $g$ and $h$ respectively, such that  $g_i=h_i$ for $1\le i \le q$, $g_i$ and $h_i$ are in the same vertex group for $q< i \le q+p$, and 
	$$h^{-1}g\equiv h^{-1}_m\dots h^{-1}_{q+p+1} (h^{-1}_{q+1}g_{q+1})\dots(h^{-1}_{q+p}g_{q+p})g_{q+p+1}\dots g_n.$$ 
	Furthermore, for each $1\le j \le p$, $h_{q+j}$ and $g_{q+j}$ are in $G(\st(q+i))$ for all $1\le i\le p$.

\end{lemma}

\begin{proof}
	Suppose $g$ and $h$ are elements of $G(\Gamma)$, $m=|h|_r$ and $n=|g|_r$.
	Let us find a reduced word representative of $h^{-1}g$ by concatenating the reduced words of $h^{-1}$ and $g$ and then reducing by syllable shuffling, cancelation, and amalgamation. 

	First, there is a maximal collection of syllables which cancel.
	Suppose $g_i$ cancels when reducing $h^{-1}g$ and there is a $g_j$ with $j<i$  which does not cancel.
If $G(i)=G(j)$, then from Remark \ref{remk}, there must be a $g_p\not\in G(\st(v))$ between $g_i$ and $g_j$. 
But $g_j$ blocks $g_p$ from cancelling, which in turn blocks $g_i$ from cancelling, a contradiction.
	So necessarily $G(i)\ne G(j)$. 
	Now if $g_j$ was not in $G(\lk(i))$ then $g_j$ would block $g_i$ from shuffling adjacent to its inverse and cancelling, a contradiction to our hypothesis.
Hence $g_j$ is in $G(\lk(i))$ and $g_i$ commutes with $g_j$.
	Hence each syllable which cancels must be able to shuffle in its word to a point where every syllable to its left will also cancel.
	From this we may assume, after rearranging, that there is a natural number  $0\le q \le \min(m,n)$ such that $g_i=h_i$ for all $i\le q$.
	In other words we assume $g$ and $h$ begin with the same reduced word $k\equiv g_1\dots g_{q}$, which may be the empty word.

	Next, some syllables may multiply non-trivially.
	Say $g_i\ne h_j$, but they are from the same vertex group $G_v$ and we are able to shuffle them next to each other when reducing $h^{-1}g$.	
	Then $g_i$ and $h_j$ must be able to shuffle to the beginning of the reduced words $k^{-1}g\equiv g_{q+1}\dots g_n$ and $k^{-1}h\equiv h_{q+1}\dots h_m$ respectively.
	If there is another such pair $g_r\ne h_s\in G_w$ which can be amalgamated, then either $g_i$ and $h_s$ or $g_r$ and $h_j$ must shuffle, implying that $G_v$ and $G_w$ commute.
	Hence all such syllables must mutually commute and there can be at most $\kappa$ pairs of syllables which multiply non-trivially (recall $\kappa$ is the maximal number of mutually adjacent vertices in $\Gamma$). 
	After shuffling, we may write $h^{-1}g$ as the product
	$$h^{-1}g\equiv h^{-1}_m\dots h^{-1}_{q+p+1} (h^{-1}_{q+1}g_{q+1})\dots(h^{-1}_{q+p}g_{q+p})g_{q+p+1}\dots g_n$$ for some $0\le q\le\min(m,n)$ and $p\le\kappa$. 
	To summarize, there are reduced words $g_1\dots g_n$ and $h_1\dots h_m$ representing $g$ and $h$ respectively such that $g_i=h_i$ for all $i\le q$ and the syllables $h_{q+j}$ and $g_{q+j}$  are in the same vertex group for $1\le j\le p$.
	Further, $h_{q+j}$ and $g_{q+j}$ are in $G(\st(q+i))$ for all $1\le i\le p$.
\end{proof}

We say $g$ \textit{ends with} $G_v$ if there is a reduced word representing $g$ such that the last syllable is in $G_v$.
	Likewise, $g$ \textit{begins with} $G_v$ if the first syllable of some reduced word representing $g$ is in $G_v$.
	Note that if $g \equiv g_1\dots g_v g_w$, and $g_v$ and $g_w$ commute ($v$ and $w$ are adjacent), then $g$ ends with both $G_v$ and $G_w$. 

	\begin{lemma}\label{initial}
	Suppose that $g,h\in G(\Gamma)$, $v\in \Gamma$, and $h$ does not end with $G(\lk(v))$.
       If $h^{-1}$ begins with $G_v$ but $h^{-1}g$ does not begin with $G_v$, then there is a reduced word $g_1\dots g_n$ representing $g$ such that $h\equiv g_1\dots g_m$ and $g_m\in G_v$ for some $m\le n$.
       Similarly, if $h^{-1}g$ begins with $G_v$ but $h^{-1}$ does not begin with $G_v$, then there is a reduced word $g_1\dots g_n$ representing $g$ such that $h\equiv g_1\dots g_m$ and $g_{m+1}\in G_v$ for some $m<n$.
\end{lemma}
\begin{proof}
Suppose $g,h \in G(\Gamma)$, $v\in \Gamma$, and $h$ does not end with $G(\lk(v))$.
By the previous lemma, we may find reduced words $g\equiv g_1\dots g_n$ and $h\equiv h_1\dots h_m$ such that $g_1\dots g_{q}=h_1\dots h_{q}$ and
$$h^{-1}g\equiv h^{-1}_m\dots h^{-1}_{q+p+1} (h^{-1}_{q+1}g_{q+1})\dots(h^{-1}_{q+p}g_{q+p})g_{q+p+1}\dots g_n$$ for some $q\le\min(m,n)$ and $p\le\kappa$.

Let us start with the first assertion and assume that
$h^{-1}$ begins with $G_v$ but $h^{-1}g$ does not begin with $G_v$.
To show that $h$ is an initial subword of $g$, we need to prove that $q=m$. 
In other words that all syllables of $h$ will cancel with syllables of $g$.
Since $h^{-1}$ begins with $G_v$ there is a syllable $h^{-1}_i\in G_v$ of $h^{-1}$ which can shuffle to the beginning of $h^{-1}$ and so $h_{j}^{-1}\in G(\lk(v))$ for all $i<j\le m$.  
This implies that $h$ ends with $G(\lk(v))$, a contradiction to our hypothesis, unless $i=m$.
Since $h^{-1}g$ does not start with $G_v$, $h_i^{-1}$ must cancel in $h^{-1}g$, so $m= i\le q\le \min(m,n)$.
Hence $q=m$, $p=0$, and $h\equiv g_1\dots g_q$.
Further, since $h$ ends with $G_v$, we can assume $i=q=m$ and thus have $g_m\in G_v$. 

For the second assertion, suppose $h^{-1}g$ begins with $G_v$ and $h^{-1}$ does not begin with $G_v$.
Again we wish to show that $q=m$, so suppose not and $q<m$.
Then $h_m^{-1}$ is either a syllable or part of an amalgamated syllable of $h^{-1}g$.
Since $h^{-1}g$ begins with $G_v$, there is some syllable $z\in G_v$ of $h^{-1}g$ which can shuffle to the beginning of $h^{-1}g$.  
If $z$ is $h_m^{-1}$, then $h^{-1}$ begins with $G_v$, a contradiction.
If not, then $z$ must shuffle past $h_m^{-1}$ and $h$ ends with $G(\lk(v))$,  a contradiction.
Hence $h_m^{-1}$ cannot be part of a syllable of $h^{-1}g$, so $q=m$, $p=0$, and $h=g_1\dots g_q$.
Also, since $h^{-1}g\equiv g_{q+1}\dots g_n$ and begins with $G_v$, we may assume $g_{m+1}\in G_v$.
\end{proof}

	The \textit{d-tail} of an element $g \in G(\Gamma)$ is the set of syllables which are in the last (rightmost) $d$ syllables of some reduced word representing $g$. 
	The $d$-tails of $g$ and $h$ contain all of the information in $h^{-1}g$ when $|h^{-1}g|_r=d$:
	\begin{lemma} \label{cancel}
			Suppose that $g$ and $h$ are elements of  $G(\Gamma)$ and $d=|h^{-1}g|_r$.
			Then each syllable of $h^{-1}g$ is either a syllable from the $d$-tail of $g$ or $h$, or an amalgamation of syllables from the $d$-tails of $g$ and $h$.
	\end{lemma}
	\begin{proof}
		Let $g$ and $h$ be elements of $G(\Gamma)$ and define $d:=|h^{-1}g|_r$.
		If all the syllables of $g$ and $h$ are in their respective $d$-tails, then the lemma is trivially proved.
	       As the argument is symmetric in $g$ and $h$, we may assume that $g$ has syllables outside its $d$-tail.
		We now show that all syllables of $g$ outside the $d$-tail of $g$ must cancel with a syllable of $h$ in the product $h^{-1}g$.

		Let $g\equiv g_1\dots g_n$, and $g_i$ be a syllable not in the $d$-tail of $g$.
		Let $S$ be the set of syllables to the right of $g_i$ in $g$ that cannot be shuffled to the left of $g_i$.
		Note that $S$ is independent of the reduced word representing $g$, and by hypothesis there are at least $d$ syllables in $S$.	
		Hence after shuffling and relabeling we may assume $g_j\in S$ for all $j>i$ and $i<n-d$.

		Now suppose $g_i$ does not cancel with a syllable of $h$ in the reduced word representing $h^{-1}g$. 
		Then at least one of the syllables $g_j$ in $S$ must cancel with a syllable $h_p$ of $h$, otherwise $|h^{-1}g|_r$ would be greater than $d$.
		This implies that $h_p$ must shuffle past $g_i\dots g_{j-1}$, which happens if and only if $g_j$ can shuffle past $g_i\dots g_{j-1}$.
		In particular, $g_j$ must shuffle to the left of $g_i$, but this is a contradiction, as $g_j$ was supposed to be in $S$.
		Thus $g_i$ must cancel with some syllable of $h$ when reducing $h^{-1}g$.
	\end{proof}

	In the case of a free product, the $d$-tail of $g$ is just the last $d$ syllables of $g$.
	For general graph products there may be more than $d$ syllables in the $d$-tail and we have the following bound on its size:
	\begin{lemma} \label{dtailsize}
		Let $\kappa$ be the maximal number of mutually adjacent vertices in $\Gamma$.
		Then the maximum number of syllables in the $d$-tail of any group element is $d\kappa$.
	\end{lemma}

	\begin{proof}
		Let $g\in G(\Gamma)$ and $n=|g|_r$.
		We prove this by induction on $1<d\le n$.
		For the base case, the $1$-tail of $g$ is just the set of final syllables of  $g$, which must all commute with each other.
		Hence the largest the $1$-tail could be is $\kappa$.

		Now let $d$ be an natural number between 1 and $n$ and assume for induction that the $(d-1)$-tail of $g$ has cardinality at most $(d-1)\kappa$.
		Split the $d$-tail of $g$ into two parts:
		the $(d-1)$-tail of $g$ and the $d$-tail$\setminus$ $(d-1)$-tail of $g$.
		
		Suppose $g_i$ is in the $d$-tail$\setminus$ $(d-1)$-tail of $g$.
		Then we may find a reduced word representing $g$ in which $g_i$ is the $(n-d)^\text{th}$ syllable, and cannot be shuffled further towards the end of the word.
		If $g_j$ is any other syllable in the $d$-tail$\setminus$ $(d-1)$-tail of $g$, then $g_j$ is somewhere to the left of $g_i$ and would have to shuffle past $g_i$ in order to rewrite the word with $g_j$ as the  $(n-d)^\text{th}$ syllable. 
		Hence $g_j$  must be in $G(St(i))$.
		This is true of any syllable in the $d$-tail$\setminus$ $(d-1)$-tail, so all such syllables commute, and there can be at most $\kappa$ of them.
		Combining this with the induction hypothesis implies that the cardinality of the $d$-tail of any group element is at most $d\kappa$.
	\end{proof}

\section{The wall space}
For the main construction we need bounds on the characteristic function of words with reduced word length $d$ in a graph product.
To do this we find a metric space that reflects the intergroup geometry of the vertex groups of a graph product.
We then use properties of this metric space to supply us with bounds.
For a free product of groups, this space is a tree.
For a direct product of $n$ groups, it is an $n$-dimensional cube, as every word has at most $n$ syllables. 
For a general graph product this intergroup geometry is a CAT(0) cube complex---cubes where groups commute and trees where they do not.

A \textit{CAT(0) cube complex} is a metric space of non-positive curvature made by gluing cubes of various dimensions together according to a few constraints.
CAT(0) cube complexes are known for their nice combinatorial properties
and are intimately related to median algebras and partially ordered with complement (poc) sets.
In this paper, we shall realize our CAT(0) cube complex by constructing its associated wall space, a geometric poc set.
For a detailed treatment of CAT(0) cube complexes, see Bridson and Haefliger \cite{BriHae99}. 
For their connection to median algebras, poc sets, and wall spaces, see Roller \cite{Roller98}  or Chatterji and Niblo \cite{ChNi05}.

First, some background:
a \textit{wall} on a set $X$ is a partition of $X$ into a half-space $w$ and its complement $w^c$.
We will refer to a wall $(w,w^c)$ by its preferred half-space $w$.
A wall $w$  \textit{separates} two points in $X$ if one point lies in the half-space $w$ and the other in $w^c$.
A \textit{wall space} is a space $X$ and a collection $W$ of walls on $X$ such that, for any $x,y\in X$, finitely many walls in $W$ separate $x$ and $y$. 
The \textit{wall distance} $d_W(x,y)$ is defined as the number of walls separating $x$ and $y$.
Two walls $w$ and $u$ \textit{cross} (symbolically $w\perp u$) if their four possible intersections $w\cap u$, $w^c \cap u$, $w \cap u^c$, and $w^c \cap u^c$ are all nonempty.
Obviously containment of one half-space in another ($w \subset u^c$ or $w \subset u$) implies the two walls do not cross, and we call the walls \textit{parallel} ($w\parallel u$).

Chatterji and Niblo \cite{ChNi05} and Nica \cite{Nica04}, both based on a construction of Sageev \cite{Sageev95}, have demonstrated that every wall space generates a cube complex.  
Let us paraphrase the result we need (Theorem 3 from \cite{ChNi05}):

\begin{theorem}\label{chat}
	Let $G$ be a discrete group acting on a space with walls $Y$ and let $W$ denote the set of walls.
	Then there exists an action of $G$ on a CAT(0) cube complex $X$ such that for each $y\in Y$, there exists an $x \in X$ such that $d_X(gx,x) = d_W(gy,y)$.
	If the maximum number of mutually crossing walls in $W$ is finite, then $X$ is finite dimensional.
\end{theorem} 

We now construct a set of walls on $G(\Gamma)$ such that the wall distance $d_\mathbb{W}(g,h)$ is proportional to the reduced word length $|h^{-1}g|_r$.
The CAT(0) cube complex generated by this wallspace is a generalization of Ruane and Witzel's cube complex for graph products of finitely generated abelian groups \cite{RuWi13}. 
It is also a natural extension of the cube complex known for Coxeter groups as found in Davis \cite{Davi08}. 
Further connections between CAT(0) cube complexes and graph products can be found in Davis \cite{Davi94}, Meier \cite{Meier96}, or Januszkiewicz and Swiatkowski \cite{JanSwi01}.  
For a completely disconnected graph, this cube complex is (the rooted) barycentrically subdivided Bass-Serre tree of the free product of the vertex groups.

For each $v\in \Gamma$, define 
$$\hw{v} := \{g\in G(\Gamma)\;|\; \exists g_v \in G_v \text{ such that } |g_vg|_r < |g|_r\}.$$ 
This is the set of elements of $G(\Gamma)$ which begin with $G_v$:  
if $g\in\hw{v}$ then there exists a $g_v \in G_v$ such that $|g_vg|_r<|g|_r$ and $g_v$ must cancel with a syllable of $g$ in any reduced word representing $g_vg$.
This means $g_{v}^{-1}$ is a syllable of $g$ and can be shuffled to the beginning of $g$, hence $g$ begins with $G_v$.
On other hand, if $g\equiv g_1\dots g_n$, and $g_1 \in G_v$, then $|g_1^{-1}g|_r = |g_2\dots g_n|_r < |g|_r$ and $g$ is in $\hw{v}$.
Similarly $\hwc{v}$, the complement of $\hw{v}$, is the set of elements of $G(\Gamma)$ that do not begin with $G_v$: 
$$\mathbb{A}_v^c:=  \{g \in G(\Gamma)\;|\; \forall g_v \in G_v,\quad  |g_vg|_r \ge |g|_r\}.$$ 
The pair $(\hw{v},\hwc{v})$ form a wall on $G(\Gamma)$.
Further, using the natural left multiplication action of $G(\Gamma)$ on subsets of $G(\Gamma)$, we have for each $k \in G(\Gamma)$ and $v \in \Gamma$ a wall $(k\hw{v},(k\hw{v})^c)$ whose half-spaces are
$$k\mathbb{A}_v = \{kg\in G(\Gamma)\;|\; \exists g_v \in G_v\text{, } |g_vg|_r < |g|_r\} = \{g\in G(\Gamma)\;|\; \exists g_v \in G_v\text{, } |g_vk^{-1}g|_r < |k^{-1}g|_r\}$$
and 
$$(k\hw{v})^c = k\hwc{v} = \{g\in G(\Gamma)\;|\; \forall g_v \in G_v\text{, } |g_vk^{-1}g|_r \ge |k^{-1}g|_r\}.$$
Our set of walls $\mathbb{W}$ is the collection of all walls of this form.  
I.e., $$\mathbb{W} := \{k\mathbb{A}_v \text{, } \forall k\in G(\Gamma)\text{, }\forall v\in \Gamma\}.$$

Note that we unambiguously refer to the wall $(k\hw{v},k\hwc{v})$ by the preferred half-space $k\hw{v}$. 
This is because the action of $G(\Gamma)$ on $\mathbb{W}$ sends preferred half spaces to preferred half spaces:
\begin{lemma}For any $k\in G(\Gamma)$, and $v, w \in \Gamma$, $k\hw{v} \ne \hwc{w}$.
\end{lemma}
\begin{proof}
	Suppose $k\equiv k_1\dots k_n$ is an element of $G(\Gamma)$ and $v$ and $w$ are vertices of $\Gamma$.
	First suppose that $k$ ends with $G_v$, so without loss of generality $k_n\in G_v$. 
	Then, for any $g_w \in G_w$ which is not a syllable of $k$, we have $|k_nk^{-1}g_w|_r=|k^{-1}_{n-1}\dots k^{-1}_1 g_w|_r < |k^{-1}g_w|_r$, and hence $g_w \in k\hw{v}$.
	But $g_w$ is not in $\hwc{w}$, since $g_w$ starts with $G_w$.
	Hence $k\hw{v} \not\subset \hwc{w}$.

	Now suppose that $k$ does not end with $G_v$. 
	Then $e$ is in $\hwc{w}$ but not in $k\hw{v}$. 

	In either case we have that $k\hw{v} \ne \hwc{w}$ and the lemma is proved.
\end{proof}

The following proposition describes some of the structure of $\mathbb{W}$.

\begin{proposition}\label{facts} For $\mathbb{W}$ as above, $k \in G(\Gamma)$, and $v, w \in \Gamma$, we have the following:
	\begin{enumerate}
		\item $k\hw{v}$ and $h\hw{w}$ cross if and only if $h^{-1}k\hw{v}$ and $\hw{w}$ cross.
		\item Suppose $v$ and $w$ are adjacent.
			Then $k\hw{v}$ and $k\hw{w}$ cross and for all $h_w \in G_w$ we have $kh_w\hw{v} = k\hw{v}$.
		\item Given a wall $k\hw{v} \in \mathbb{W}$, there is an unique minimal (with respect to $|\cdot|_r$) element $h\in G(\Gamma)$ such that $h\hw{v} = k\hw{v}$.
		\item If $v$ and $w$ are non-adjacent or equal and $k$ is minimal as in (iii), then $k\hw{v}$ is parallel to $\hw{w}$.
\end{enumerate}
\end{proposition}

\begin{proof}
Part (i) follows from invariance built into the walls:
for any wall $k\hw{v}$, an element $g$ is in $k\hw{v}$ if and only if $h^{-1}g\in h^{-1}k\hw{v}$.
Hence, given another wall $h\hw{w}$, an element $g$ is in $k\hw{v}\cap h\hw{w}$ if and only if $h^{-1}g\in h^{-1}k\hw{v}\cap \hw{w}$. 
	Likewise the other three intersections $k\hw{v}\cap h\hwc{w},$ $k\hwc{v}\cap h\hw{w},$ and $k\hwc{v}\cap h\hwc{w}$ are non-empty if and only if $h^{-1}k\hw{v}\cap \hwc{w}$, $h^{-1}k\hwc{v}\cap \hw{w}$, and $h^{-1}k\hwc{v}\cap \hwc{w}$ are non-empty respectively.
	Thus if $k\hw{v}$ and $h\hw{w}$ cross, the eight intersections above are all non-empty and $h^{-1}k\hw{v}$ and $\hw{w}$ cross.

	For (ii), suppose that $k\hw{v}$ is a wall and $v$ and $w$ are adjacent vertices of $\Gamma$. 
	We first prove that $ k\hw{v}$ and $k\hw{w}$ cross. 
	Note for any $g_v \in G_v$ and $g_w\in G_w$ we have the following containments: 
	$$k^{-1} \in  k\hwc{v}\cap k\hwc{w},\; 
	k^{-1}g_v \in  k\hw{v}\cap k\hwc{w},\;
	k^{-1}g_w \in  k\hwc{v}\cap k\hw{w},\text{ and }
	k^{-1}g_vg_w \in  k\hw{v}\cap k\hw{w}.$$
	As these four intersections are all non-empty, by definition $k\hw{v}$ and $k\hw{w}$ cross.

	For the second part, suppose $h_w\in G_w$.
	Since $h_w$ commutes with elements of $G_v$, the group element $h_w^{-1}k^{-1}g$ begins with $G_v$ if and only if $k^{-1}g$ begins with $G_v$. This implies that $kh_w\hw{v}=k\hw{v}$.

	To prove part (iii), let  $k\hw{v}$ be a wall.
	If $k$ ends with $G(\lk(v))$, then we may write $k\equiv k_1\dots k_n$ where $k_n\in G_w$ and $v$ and $w$ are adjacent.
	By part (ii) we see that $k\hw{v}=k_1\dots k_n\hw{v} = k_1\dots k_{n-1}\hw{v}$.
	We may repeat this process until we have an $h=k_1\dots k_j$ which no longer ends with $G(\lk(v))$ and $k\hw{v} = h\hw{v}$.

	Next suppose $h\hw{v}$ is a wall such that $h\hw{v} = \hw{v}$.
	This imposes a strong constraint on $h$: for any $g\in G(\Gamma)$, $g$ begins with $G_v$ if and only if $h^{-1}g$ begins with $G_v$.
	In particular, $e$ does not begin with $G_v$, so $h^{-1}e=h^{-1}$ must not begin with $G_v$.
	Let $g_v$ be any element of $G_v$.
	Then $g_v$ begins with $G_v$ which implies that $h^{-1}g_v$ must begin with $G_v$.
	But $h^{-1}$ does not begin with $G_v$, hence $g_v$ must be able to shuffle to the beginning of $h^{-1}$. 
	In other words, $h^{-1}g_v=g_vh^{-1}$ and every syllable of $h$ must be in $G(\lk(v))$. 

	Finally, given two equivalent walls $k\hw{v}=h\hw{v}$, both reduced as above, we have that $h^{-1}k\hw{v} = \hw{v}$, so every syllable of $h^{-1}k$ must be in $G(\lk(v))$.
	By Lemma \ref{hg}, we may find reduced words $k\equiv k_1\dots k_n$ and $h\equiv h_1\dots h_m$ such that $k_1\dots k_{q}=h_1\dots h_{q}$ and
	$$h^{-1}k\equiv h^{-1}_m\dots h^{-1}_{q+p+1} (h^{-1}_{q+1} k_{q+1})\dots(h^{-1}_{q+p}k_{q+p})k_{q+p+1}\dots k_n$$ for some $q\le\min(m,n)$ and $p\le\kappa$. 

Since reduction implies that $h$ does not end with $G(\lk(v))$, $q\ge m$.
Similarly, since $k$ does not end with $G(\lk(v))$, $q\ge n$.
Since $q\le \min(m,n)$, we have $q=m=n$, and $h^{-1}k=e$.
Hence if $k$ and $h$ are representatives of the same wall and reduced as above, then $k=h$.

For part (iv), suppose that $v$ and $w$ are non-adjacent or equal and $k\hw{v}$ is a wall.  
	We wish to show $k\hw{v}\parallel \hw{w}$.
	If $k$ is not a reduced as in part (iii), then we could reduce it to $h\hw{v}$, and $h\hw{v}\parallel \hw{w}$ if and only if $k\hw{v}\parallel \hw{w}$. 
	So without loss of generality, $k$ does not end with $G(\lk(v))$.
      
	If $k=e$, then if $v=w$, $\hw{v}$ and $\hw{w}$ are parallel since they are equal.
	If $v$ and $w$ are not adjacent, then $\hw{v}\cap\hw{w}$ is empty, as no word can begin with both $G_v$ and $G_w$.
	We now assume $0<m=|k|_r$.

	First suppose $k$ ends with $G_v$ and $g\in k\hwc{v}$.
	Then $k^{-1}g$ does not begin with $G_v$ and by Lemma \ref{initial}, $g\equiv g_1\dots g_n$ such that $k\equiv g_1\dots g_m$ and $g_m\in G_v$.
	In other words, $g$ must begin with $k$.
	If $k$ begins with $G_w$, then $g$ must begin with $G_w$ and $k\hwc{v}\cap\hwc{w}$ must be empty.
	If $k$ does not begin with $G_w$ but $g$ does, then some syllable $g_i$ with $i>m$ is in $G_w$ and $g_j\in G(\lk(w))$ for all $j<i$.
	But when $v$ and $w$ are non-adjacent, $g_m$ is not in $G(\lk(w))$, a contradiction.
        When $v=w$, by Remark \ref{remk}, there is some $m<j<i$ such that $g_j\not\in G(\st(w))$, again a contradiction. 
	Hence $g$ cannot begin with $G_w$ and $k\hwc{v}\cap\hw{w}$ must be empty.

	Now suppose $k$ does not end with $G_v$ and $g\in k\hw{v}$.
	Then $k^{-1}g$ begins with $G_v$ and by Lemma \ref{initial} $g$ must begin with $k\equiv g_1\dots g_m$ and $g_{m+1}\in G_v$.
	Following from the arguments in the previous paragraph, but replacing $g_m$ with $g_{m+1}$, we see that
if $k$ begins with $G_w$, then $g$ begins with $G_w$, and $k\hw{v}\cap\hwc{w}$ must be empty.
If $k$ does not begin with $G_w$, then $g$ cannot begin with $G_w$ and $k\hw{v}\cap\hw{w}$ must be empty.

 As these cases exhaust all the possibilities for $k$, we have shown that $k\hw{v}$ is parallel to $\hw{w}$ whenever $v$ and $w$ are not adjacent or $v=w$.
\end{proof}

We call the unique minimal wall representative $h$ of Proposition \ref{facts} (iii) a \textit{reduced wall representative}.

\begin{corollary}
	There is an uniform bound on the size of a set of mutually crossing walls of $\mathbb{W}$.
\end{corollary}

\begin{proof}
	Recall that $\kappa$ is the maximum number of mutually adjacent vertices of $\Gamma$.
	Combining Proposition \ref{facts} (i) and (iv), we see that two walls cross  only if they have adjacent vertices.
	Since a maximum of $\kappa$ vertices are mutually adjacent, at most $\kappa$ walls mutually cross.
\end{proof}

\begin{lemma}
	For $\mathbb{W}$ as above, the wall distance between $e$ and $g \in G(\Gamma)$ is $d_W(e,g) = 2|g|_r$.
\end{lemma}

\begin{proof}
	The case when $g=e$ is trivial, so we assume $g\ne e$ and our task is to count the number of walls which separate $e$ and $g$. 
	Let $g\equiv g_1g_2\dots g_m$ and suppose $k$ is a reduced wall representative such that $k\hw{v}$ separates $g$ and $e$.
Recall that $k\hw{v}$ separates $e$ and $g$ if either $e$ is in $k\hw{v}$ and $g$ is in $k\hwc{v}$ or $g$ is in $k\hw{v}$ and $e$ is in $k\hwc{v}$.

	If $e$ is in $k\hw{v}$ and $g$ is in $k\hwc{v}$ then $k^{-1}$ begins with $G_v$ and $k^{-1}g$ does not. 
	By Lemma \ref{initial}, $k\equiv g_1\dots g_n$ and $g_n\in G_v$ for some $n\le m$.
	If $g$ is in $k\hw{v}$ and $e$ is in $k\hwc{v}$,
	then $k^{-1}$ does not begin with $G_v$ and $k^{-1}g$ does begin with $G_v$. 
	Appealing again to Lemma \ref{initial}, $k$ must have the form $k\equiv g_1\dots g_n$ and $g_{n+1}\in G_v$ for some $n<m$.
	From these two statements we see that for each syllable $g_i$ of $g$, with $v\in \Gamma$ such that $g_i\in G_v$, the walls $g_1\dots g_i\hw{v}$ and $g_1\dots g_{i-1}\hw{v}$ separate $e$ and $g$ and are the only walls to do so.
	Hence $d_W(e,g) = 2|g|_r$.
\end{proof}

\begin{lemma}
	The wall distance is $G(\Gamma)$-invariant: $d_\mathbb{W}(h,hg) = d_\mathbb{W}(e,g)$ for all $g,h \in G(\Gamma)$.
\end{lemma}
\begin{proof}
	This follows from the construction of the walls:  For any $g$ and $h\in G(\Gamma)$, $k\hw{v}$ separates $g$ and $e$ if and only if $hk\hw{v}$ separates $hg$ and $h$, since $g \in k\hw{v}$ if and only if $hg\in hk\hw{v}$.
\end{proof}

We may now apply Theorem \ref{chat} to $\mathbb{W}$ to realize a finite dimensional CAT(0) cube complex $X$ on which $G(\Gamma)$ acts, with $d_X(gx_0,hx_0)=2|h^{-1}g|_r$ for some $x_0\in X$.
We shall return to this distance function in Section 5.

\section{Extending functions to free products}

In this section we shall prove two special cases of the main theorem:
if two discrete groups $A$ and $B$ are either amenable or weakly amenable, then $A\ast B$ is weakly amenable.
In doing so we shall develop the machinery used in the final proof.
The case for amenable factors was originally proved by Bozejko and Picardello \cite{BoPi93} in 1993. 
In 2006 Ricard and Xu \cite{RiXu06} proved the weakly amenable case in the more general context of free products of $C^{\ast}$-algebras.
The latter proof was analytic in nature. 
In this paper we give a constructive proof which is generalizable to graph products.
The key step in this construction is the extension of positive definite or completely bounded (with cb-norm close to  1) functions from the factor groups to their free product.

Let us motivate this extension.
Consider $\Z\ast \Z = \F_2$, the free group on two generators.
By Schoenberg's theorem, $\f_{\Z}(n):=e^{-|n|}$ is a positive definite function on $\Z$.
If we wish to extend $\f_\Z$ to a function on the free group, there is a natural choice: the function $\bF(g):= e^{-\|g\|}$.
It is positive definite (Haagerup gives a proof in \cite{Haag78}), the restriction of $\bF$ to either factor group is $\f_\Z$, and
for a word $g=a^{p_1}b^{q_1}\dots a^{p_n}b^{q_n}$ in $\F_2$,
\begin{equation*}
	\begin{split}
		\bF(g) 
		&= e^{-\|g\|}
		= e^{-\left(\sum_1^n |p_i|+|q_i|\right)} 
		\\ &= e^{-|p_1|}e^{-|q_1|}\dots e^{-|p_n|}e^{-|q_n|}
		\\ &= \f_\Z(p_1)\f_\Z(q_1)\dots \f_\Z(p_n)\f_\Z(q_n).
	\end{split}
\end{equation*}
We see that $\bF(g)$ is just the product of $\f_\Z$ applied to each syllable (under the isomorphism $a^{n}\mapsto n$).
For discrete groups other than $\Z$ we would like something similar.
\subsection{The positive definite case.}
Let $G= A \ast B$ and let $\f_A$ and $\f_B$ be positive definite functions on the factor groups with $\f_A(e) = 1 = \f_B(e)$.
Recall that since $\f_A$ is positive definite, we may represent it as $\f_A(a^{-1}a') = \langle\alpha_A(a'),\alpha_A(a)\rangle$, where $\scr{H_A}$ is a Hilbert space and $\alpha_A$ is a map of $A$ into $\scr{H_A}$ with $\|\alpha_A(a)\|=1$ for all $a \in A$.
We have a similar Hilbert space-valued map $\alpha_B$ associated with $\f_B$.
Guided by our example above, for $g = a_1b_1\dots a_nb_n$, we would like to define our extension as
\begin{equation*}
	\begin{split}
		\langle \alpha(g),\alpha(e)\rangle  
		\stackrel{?}{=} \bF(g)
		&:=\f_A(a_{1})\f_B(b_{1})\dots \f_A(a_{n})\f_B(b_{n})
		\\ &=\langle \alpha_A(a_{1}),\alpha_A(e)\rangle \langle \alpha_B(b_{1}),\alpha_B(e)\rangle \dots \langle \alpha_A(a_{n}),\alpha_A(e)\rangle \langle \alpha_B(b_{n}),\alpha_B(e)\rangle 
	\end{split}
\end{equation*} 
The questionable equality would characterize $\bF$ as positive definite.
We can realize such an equality if the Hilbert space in which we take the inner product on the left-hand side is a tensor product of the Hilbert spaces used on the right-hand side. 
This product must be over some nice index set that isolates syllables of words in $G$.

Let $T$ be the set of left cosets of $A$ and $B$ ($T\cong G/A \sqcup G/B$).
For each $t\in T$, define the Hilbert space $\scr{H}_t$ and its vacuum vector $\omega_t$ as either $\scr{H}_A$ and $\alpha_A(e)$  or $\scr{H}_B$ and $\alpha_B(e)$, depending on if $t$ is a left coset of $A$ or $B$, respectively.  
Let $\scr{\widehat H}:= \bigotimes_T \scr{H}_t$ be the infinite tensor product of the $H_t$ over $T$ as described in Section \ref{hilbert}.

Since $G$ acts by left multiplication on $T$, it induces an unitary action $\pi$  on $\scr{\widehat H}$ by permuting indices, namely
$\pi_{g} \xi_{t} = \xi_{gt}$, where $\xi_{t}$ is the $t$ component of the tensor $\xi \in \scr{\widehat H}$.

From now on we shall write, for $g$ in $G$, $g=g_1g_2\dots g_n$, where the representation is understood to be an alternating word in the factor groups $A$ and $B$.
For a syllable $g_i$, we write $G(i)$ for the factor group containing $g_i$
and write $\alpha(g_i)$ for $\alpha_{G(i)}(g_i) \in \scr{H}_{G(i)}$.
For such a word $g = g_1g_2\dots g_n$ in $G$, define 
$$\widehat\alpha(g) := \alpha(g_1)_{G(1)}\otimes \alpha(g_2)_{g_1G(2)} \otimes \alpha(g_3)_{g_1g_2G(3)} \otimes \cdots 
\otimes \alpha(g_n)_{g_1\dots g_{n-1}G(n)}.$$
This is a map from $G$ to $\scr{\widehat H}$, and the kernel defined by the inner product $\langle\widehat\alpha(g),\widehat\alpha(e)\rangle$ is exactly the extension we were looking for: 
\begin{equation*}
	\begin{split}
		\langle\widehat\alpha(g),\widehat\alpha(e)\rangle
		&=\langle \alpha(g_1),\alpha(e)\rangle_{G(1)} 
		\langle\alpha(g_2),\alpha(e)\rangle_{g_1G(2)} \cdots
		\langle\alpha(g_n),\alpha(e)\rangle_{g_1\dots g_{n-1} G(n)}
		\\	&= \f_{G(1)}(g_1)\f_{G(2)}(g_2)\cdots \f_{G(n)}(g_n).
	\end{split}
\end{equation*}

Our choice of $T$ gives us a geometric interpretation of this construction.
We can construct a graph $X$ using $T$ as the vertex set, connecting two cosets with an edge if they share an element.
This implies that no two cosets of the same factor group share an edge.
This graph is a tree (called the Bass-Serre tree) and each edge corresponds uniquely to a group element.
There is an map of $G$ into the paths of $X$, sending each $g$ to the path from $G(1)$ to $g_1\dots g_{n-1} G(n)$.
Each edge along the path corresponds to a subword $g_1\dots g_i$ of $g$.
Given two group elements $g$ and $h$, the paths of $g$ and $h$  overlap if and only if $g$ and $h$ begin with the same subword.

Now suppose that $h=h_1\dots h_m\in G$.
Notice that $\widehat\alpha(g)$ has non-vacuum entries only for indices which are along  the path of $g$ in $X$, including endpoints, likewise with $\widehat\alpha(h)$.
When we take an inner product $\langle \widehat\alpha(g),\widehat\alpha(h)\rangle$, we get values other than one only on indices along the paths of $g$ and $h$ in $X$.
If the two paths overlap then $g$ and $h$ must begin with the same subword $k=g_1\dots g_i = h_1\dots h_i$.
For each syllable $g_j$ of $k$, the $g_1\dots g_{j-1}G(j)^\text{th}$ factor of the inner product $\langle\widehat\alpha(g),\widehat\alpha(h)\rangle$ is $\langle \alpha(g_j),\alpha(h_j)\rangle = \langle\alpha(g_j),\alpha(g_j)\rangle =1$. 
Hence we have 
\begin{equation*}
	\begin{split}
		\langle\widehat\alpha(g),\widehat\alpha(h)\rangle 
		&= 1\cdot \langle\alpha(g_{i+1}),\alpha(h_{i+1})\rangle 
		\cdot \left(\prod_{j>i+1}^n \langle\alpha(g_j),\alpha(e)\rangle\right)\cdot \left(\prod_{j>i+1}^m \langle\alpha(e),\alpha(h_j)\rangle\right)
		\\&= \langle\alpha(h_{i+1}^{-1}g_{i+1}),\alpha(e)\rangle\cdot \left(\prod_{j>i+1}^n \langle\alpha(g_j),\alpha(e)\rangle\right)\cdot \left(\prod_{j>i+1}^m \langle\alpha(h^{-1}_j),\alpha(e)\rangle\right)
		\\ &= \langle\widehat\alpha(h^{-1}g),\widehat\alpha(e)\rangle.
	\end{split}
\end{equation*}
A similar argument shows that $\langle\widehat\alpha(lg),\widehat\alpha(lh)\rangle = \langle\widehat\alpha(g),\widehat\alpha(h)\rangle$ for any $l \in G$, so the positive definite function $\bF(h^{-1}g) := \langle\widehat\alpha(g),\widehat\alpha(h)\rangle$  is  well defined on $G$.

Note that $\bF$ is not finitely supported regardless of whether or not $\f_A$ or $\f_B$ are.
This is as it should be---if we could construct a finitely supported positive definite function using the techniques above, our free product would be amenable! 
We now have enough to prove the first special case.
As it is a corollary of the main theorem, we only sketch the proof.
\begin{theorem}
	If $A$ and $B$ are amenable discrete groups, $A\ast B$ is weakly amenable.
\end{theorem}
\begin{proof}[sketch]
	If $A$ and $B$ are amenable, then we have two sequences $\{\f_A^n\}$ and $\{\f_B^n\}$ consisting of finitely supported positive definite functions which go to the identity as $n$ goes to infinity.
	We can use the above construction to make a sequence of positive definite functions $\{\bF^n\}$ on $A\ast B$.
	As $n$ goes to infinity, $\bF^n$ approaches the identity on each factor group, and hence on $A \ast B$, but is not finitely supported.
	In \cite{BoPi93} Bozejko and Picardello showed that $\chi_d$, the characteristic function of the $d$-sphere  of a tree, is completely bounded with  cb-norm growing polynomially as $d$ goes to infinity.  
	The growth of words in a free product is described by its Bass-Serre tree $T$, our index set above, hence we may induce finite support by truncation, taking $\sum^Ne^{-d/r}\bF^n\chi_d$.
	The exponential tames the growth of $\cb{\chi_d}$, where $r$ is chosen large enough that $e^{-|g|/r}F^n(g)$ is close to one on a given subset.
	We lose our positive definiteness but remain completely bounded with cb-norm approaching one from above as we truncate closer to infinity.
	We can thus construct a sequence of finitely supported completely bounded functions on $G$ which converge to the identity, and $G$ is weakly amenable.
\end{proof}

\subsection{The completely bounded case.} 
Now suppose that the factor groups are only weakly amenable.
Given two completely bounded functions on $A$ and $B$ with cb-norms less than $1+\varepsilon$, and $\f_A(e)=1=\f_B(e)$, we would like to extend them to a completely bounded function on $G$.
By Proposition \ref{GNS}, for each $\f$ we have a Hilbert space $\scr{H}$ with two maps $\alpha$ and $ \beta$ into $\scr{H}$, such that $\f(h^{-1}g) = \langle\alpha(g),\beta(h)\rangle$.
	Note: 	we decorate $\f$  as $\f_A$ or $\f_B$ only if the usage is ambiguous, otherwise relying on context.
	The same holds for $\alpha$, $\beta$, and $\scr{H}$.
The fact that the kernel of a completely bounded function is the inner product of some Hilbert space-valued maps gives us hope that we can mimick the positive definite construction above, with the added complications that $\alpha\ne \beta$ and $\|\alpha\|\ne 1 \ne \|\beta\|$.	

The first difficulty arises in the construction of the tensor product,
as we no longer have an obvious choice for the vacuum vectors.
As a first step, we create positive definite kernels which are close to our $\f$s.

For each element $g$ of each factor group, define the weighted average $$\omega(g):=\tfrac{\alpha(g)+\beta(g)}{2+\sqrt{2\varepsilon}} \text{ and } D(g) := \sqrt{\tfrac{1-\|\omega(g)\|^2}{2}},$$
	an additive renomalization constant associated to $\omega(g)$.

	For the factor group $A$ and any element $a\in A$, the map 
	$$a \mapsto \left\{ \begin{array}{ll} 
		&\ambient{a}{} \text{ if $a \ne e$} \\
		&\ambiente{} \text{ if $a = e$}
	\end{array}\right.$$ 
	maps $A$ into $\scr{H}_A\oplus \C^2\oplus\C^2$ and generates the (non $A$-invariant) positive definite kernel 
	$$\psi_A(a,a') := \left\langle  \ambient{a}{},\ambient{a'}{}\right\rangle.$$
	By the definition of $D(a)$, we have that $\psi_A(a,a)=\left\|\ambient{a}{}\right\|^2 = 1 \text{ for all } a\in A$.
	We have a matching $\psi_B$ on $B$ as well.

	Now that we have two positive definite kernels, we extend them to $G$ as above.
	Let $T\cong G/A \sqcup G/B$ be the set of left cosets as in the previous section.
For each $t\in T$, define $\scr{H}_t$ and its vacuum vector to be $\scr{H}$ and $\ambiente{}$ which depend on the factor group of which $t$ is a coset.
	Let $\scr{\widehat H}  := \bigotimes_T (\scr{H}_t\oplus \C^2\oplus \C^2)$.
	For $g= g_1g_2\dots g_n$, define
	\begin{multline*}
		\widehat\omega(g) := \ambient{g_1}{}_{G(1)}
		\otimes \ambient{g_2}{}_{g_1G(2)}\otimes\\
		\dots \otimes \ambient{g_n}{}_{g_1\dots g_{n-1}G(n)}.
	\end{multline*}
	This is a map from $G$ to $\scr{\widehat H}$.

	We call $\Psi(g,h):=\langle \widehat\omega(g),\widehat\omega(h)\rangle$ the \textit{ambient positive definite kernel} associated with $\f_A$ and $\f_B$.
	If $\varepsilon$ is very small, then $\omega\approx \alpha\approx \beta$, and $\Psi$ is very close to the map we want. 
	Unfortunately this kernel is not $G$-invariant.

	Recall that in the proof sketch above our final functions were of the form $\sum^n e^{-d/r} F\chi_d$.
       This implies that instead of searching for a single extension $F$ which is well defined on the group, 	we may search for a family of functions $F_d$, each well defined only for words of length $d$.
       In other words, for each $d$  we may look for a kernel which is $G$-invariant on pairs of elements $g$ and $h$ such that $|h^{-1}g|_r=d$. 
       By Lemma \ref{cancel}, if $|h^{-1}g|_r =d$ for a pair $g, h \in G$, then all but the last $d$ syllables of $g$ and $h$ cancel and hence we may try to modify our map $\widehat\omega$ on the last $d$ syllables to match our original maps $\alpha$ and $\beta$.

\begin{proposition}\label{cbext}
	Suppose that $\f_A$ and $\f_B$ are two completely bounded functions on the groups $A$ and $B$ with cb-norms less than $1+\varepsilon$, and $\f_A(e)=1=\f_B(e)$. 
	Then for each natural number $d$ there exists a completely bounded kernel $\phi_{d}$ on $A\ast B$ such that  $\phi_{d}(g,e) = \f_1(g_1)\f_2(g_2)\cdots\f_d(g_d)$ for all $g\in G$ with word length $d$.
\end{proposition}
\begin{proof}
	Let $G=A\ast B$ and $d$ be a natural number.
	For each factor group we have a completely bounded function $\f$ and some Hilbert space $\scr{H}$ with two maps $\alpha$ and $ \beta$ into $\scr{H}$, such that $\f(h^{-1}g) = \langle\alpha(g),\beta(h)\rangle$.
	Since we assume $\cb{\f} < 1+\varepsilon$, $\|\alpha\|^2$ and $\|\beta\|^2$ are bounded between $1-\varepsilon$ and $1+\varepsilon$.
	Let $\omega$, $D(g)$, and $\widehat\omega$ be defined as above.
		For every pair of elements $g$ and $h$ in the same factor group, let $$C^\alpha(g,h) := \frac{\left\langle\alpha(g),\beta(h)\right\rangle-\left\langle\alpha(g),\omega(h)\right\rangle}{D(h)}
	\text{ and } C^\beta(g,h) = \frac{\langle\alpha(h),\beta(g)\rangle - \langle \omega(h),\beta(g)\rangle}{D(h)}.$$
	These are additive renormalization constants used to correctly fit the two embeddings together.
	They are all less than $8\varepsilon^{1/4}$, as we will show in Lemma \ref{approx}.

	Define, for $g= g_1\dots g_n \in G$, 
	$$\widehat\alpha(g) := \bigotimes_{0<i\le n} \left\{
	\begin{array}{lr}
		\calpha{g_i}{}_{g_1\dots g_{i-1} G(i)}& \text{ if $n-i \le d$}\\
		\ambient{g_i}{}_{g_1\dots g_{i-1} G(i)}& \text{ otherwise}
	\end{array}\right.$$
	and likewise (note the different placement of the constants), 
	$$\widehat\beta(g) := \bigotimes_{0<i\le n} \left\{
	\begin{array}{lr}
		\cbeta{g_i}{}_{g_1\dots g_{i-1} G(i)}& \text{ if $n-i \le d$}\\
		\ambient{g_i}{}_{g_1\dots g_{i-1} G(i)}& \text{ otherwise.}
	\end{array}\right.$$
	In other words, these maps are $\varepsilon^{1/4}$-perturbations of $\alpha$ or $\beta$ on the last $d$ syllables of $g$  and is equal to $\widehat\omega$ on the rest of $g$.

	To uncompress the notation above, we have
	\begin{multline*}
		\widehat\alpha(g) =
		\ambient{g_1}{}_{G(1)}
		\otimes\cdots\\
		\otimes \ambient{g_{j-1}}{}_{g_1\dots g_{j-2}G(j-1)}
		\otimes{\calpha{g_{j}}{}}_{g_1\dots g_{j-1} G(j)}\otimes\\
		\cdots
		\otimes{\calpha{g_{n}}{}}_{g_1\dots g_{n-1} G(n)},
	\end{multline*}
	where $n-j=d$.

	Let us show that $\widehat\alpha(g)$ is bounded.
	First notice that $\widehat\alpha(g)$ has at most $d$ non-unit vectors corresponding to the last $d$ syllables of $g$.
	We have 
	\begin{equation*}
		\begin{split}
			\|\widehat\alpha(g)\| 
			&= \prod_{i=n-d}^{n}  \|\calpha{g_i}{}\| \\
			&= \prod_{i=n-d}^{n} \sqrt{ \|\alpha_i(g_i)\|^2+|C^\alpha(g_i,g_i)|^2+|C^\alpha(g_i,e)|^2}
			\\&\le \left(1+\varepsilon+128\varepsilon^{\frac{1}{2}}\right)^{\frac{d}{2}}
			\le (1+129\sqrt{\varepsilon})^{\frac{d}{2}}
		\end{split}
	\end{equation*}
	where the penultimate inequality comes from Lemma \ref{approx} and is independent of the choice of $g$. 
	We have the same inequality for $\|\widehat\beta(g)\|$.

	Our final kernel is $\phi_{d}(g,h) := \langle\widehat\alpha(g),\widehat\beta(h)\rangle$.  
	This kernel is completely bounded, with 
	\[\cb{\phi_d}\le\sup_{g,h\in G}\|\widehat\alpha(g)\|\|\widehat\beta(h)\|\le \left(1+129\varepsilon^{\frac{1}{2}}\right)^{d}.\]
	Further, for $g = g_1\dots g_d$ in $G$,
	\begin{equation*}
		\begin{split}
			\phi_{d}(g,e) &=
			\prod_{i=1}^d \left \langle \calpha{g_i}{},\ambiente{}\right \rangle_{g_1\dots g_{i-1}G(i)}
			\\ &=\prod_{i=1}^d \langle \alpha(g_i),\beta(e)\rangle
			= \prod_{i=1}^d \f(g_i),
		\end{split}
	\end{equation*}
	where the second equality follows from calculation and the definition of $C^{\alpha}(g,e)$.

\end{proof}

Note that for words $g,h \in G$ with $d=|h^{-1}g|$,
	the kernel $\phi_{d}$ constructed in the last proposition is $G$-invariant: $\phi_{d}(g,h)=\phi_{d}(h^{-1}g,e)$.
	This follows as a corollary to Lemma \ref{invariance}.

\begin{lemma}\label{approx}
	The constants	$C^\alpha(g,h)$ and $C^\beta(g,h)$ are less than $8\varepsilon^{1/4}$ for all $g,h$ in $G$.
\end{lemma}
\begin{proof}
	Recall that given two Hilbert space-valued maps $\alpha$ and $\beta$, with $\|\alpha\|^2,\,\|\beta\|^2<1+\varepsilon$, we defined $\omega(g)=\frac{\alpha(g)+\beta(g)}{2+\sqrt{2\varepsilon}}$,  
	$D(g) = \sqrt{\frac{1-\|\omega(g)\|^2}{2}}$,
	$C^\alpha(g,h) = \frac{\left\langle\alpha(g),\beta(h)\right\rangle-\left\langle\alpha(g),\omega(h)\right\rangle}{D(h)}$
	and $C^\beta(g,h) = \frac{\langle\alpha(h),\beta(g)\rangle - \langle \omega(h),\beta(g)\rangle}{D(h)}$ for all $g$ and $h$ in the same factor group.

	First observe that $\beta$ is close in norm to $\omega$: 
	\begin{equation*}
		\begin{split}
			\|\beta(g)-\omega(g)\|^2
			&=\left\|\frac{((2+\sqrt{2\varepsilon}) -1)\beta(g)-\alpha(g)}{2+\sqrt{2\varepsilon}}\right\|^2
			\\ & =\frac{\|(1+\sqrt{2\varepsilon})\beta(g)\|^2+\|\alpha(g)\|^2-2((2+\sqrt{2\varepsilon})-1)}{(2+\sqrt{2\varepsilon})^2}
			\\ &\le \frac{(1+\sqrt{2\varepsilon})^2(1+\varepsilon)+(1+\varepsilon)-(2+2\sqrt{2\varepsilon})}{(2+\sqrt{2\varepsilon})^2}
			= \frac{2\varepsilon+2\sqrt{2\varepsilon}+4}{(2+\sqrt{2\varepsilon})^2}\varepsilon
			\le\varepsilon.
		\end{split}
	\end{equation*}

	We also have that
	$\left\|\omega(g)\right\|^2
	= \frac{\|\alpha(g)\|^2+\|\beta(g)\|^2+2}{(2+\sqrt{2\varepsilon})^2}
	\le \frac{4+2\varepsilon}{(2+\sqrt{2\varepsilon})^2}$, which implies that $D(g)$ is bounded away from $0$:
	$$D(g) = \sqrt{\frac{1-\|\omega(g)\|^2}{2}}  
	\ge \sqrt{\frac{1}{2}-\frac{2+\varepsilon}{(2+\sqrt{2\varepsilon})^2}}
	=\sqrt{\frac{2\sqrt{2\varepsilon}}{(2+\sqrt{2\varepsilon})^2}} 
	\ge \frac{\varepsilon^{1/4}}{4}. $$

	Putting these together, we see that
	$$C^\alpha(g,h) 
	=\frac{\langle\alpha(g),\beta(h)-\omega(h)\rangle}{D(h)}
	\le \frac{\|\alpha(g)\|\|\beta(h)-\omega(h)\|}{D(h)} 
	\le \frac{4\sqrt{1+\varepsilon}\sqrt{\varepsilon}}{\varepsilon^{1/4}}
	\le 8\varepsilon^{1/4}. $$

	We have the same bounds for $C^\beta(g,h)$, as, by an almost identical proof, $\|\alpha(g)-\omega(g)\|^2\le \varepsilon$.
	These are coarse estimates and the constant may be improved, but they are sufficient for our purposes.
\end{proof}

We can now sketch the proof of the next theorem, inspired by Ricard and Xu's proof in \cite{RiXu06},
\begin{theorem}
	If $A$ and $B$ are weakly amenable discrete groups, then $A\ast B$ is weakly amenable.
\end{theorem}
\begin{proof}[sketch]
	If $A$ and $B$ are weakly amenable, then we have two sequences consisting of finitely supported completely bounded functions $\{\f_A^n\}$ and $\{\f_B^n\}$ which go to the identity and whose cb-norms go to one as $n$ goes to infinity.
	For each $n$ we have an ambient positive definite kernel $\Psi^n$ and, by Proposition \ref{cbext}, a family of completely bounded functions $\{\bF^n_d\}_{d\in\N}$ on $G$ with $\bF^n_d(h^{-1}g) := \phi_d(g,h)\chi_d(g,h)$ for each $d\in \N$ 
	.
	We would like to take finite sums of the form $\bF^{r,n,M}:=\sum^Me^{-d/r}\bF^n_d$.
	We can bound $\cb{\bF^{r,n,M}}$ by $\cb{\phi^n_d-\Psi^n}$ and the cb-norm of a truncated $\Psi^n$.
	The first summand is controlled by $\varepsilon_n^{1/4}$ and the second aproaches one as $M$ goes to infinity, since the infinite sum  $\sum^\infty e^{-d/r}\Psi^n\chi_d = e^{-d(\cdot,\cdot)/r}\Psi^n(\cdot,\cdot)$ is positive definite.
	Hence $\cb{\bF^{r,n,M}}$ approaches $1$ as $n$ and $M$ approach infinity.
	This comparison to a positive definite kernel is the reason why we only discuss weak amenablity with Cowling-Haagerup constant $1$.
	We then notice that, as $r$, $M$, and $n$ become sufficiently large, $\bF^{r,n,M}$ approaches the identity on the group and its cb-norm approaches one. 
	Thus we may construct an approximate identity of completely bounded functions on $G$ and $G$ is weakly amenable.
\end{proof}

\section{The main theorem}

In this section we prove that weak amenability is stable under graph products.
The first step is to generalize to graph products the above extension of completely bounded functions.
The only difference between the construction for free products and the following construction for graph products is the index set over which we build the Hilbert space.

\begin{proposition}\label{gpextension}
	Suppose $G(\Gamma)$ is a graph product over $\Gamma$, which is a graph with at most $\kappa$ mutually adjacent vertices, and, for each vertex group $G_v$ of $G(\Gamma)$, there is a completely bounded function $\f_v$ with cb-norm less than $1+\varepsilon$ such that $\f_v(e)=1$.
	Then for each natural number $d$  there exists a completely bounded kernel $\phi_d$ on $G(\Gamma)$ such that  $\phi_d(g,e) = \f_1(g_1)\f_2(g_2)\cdots\f_d(g_d)$ for all $g\in G(\Gamma)$ with $|g|_r=d$.
\end{proposition}

\begin{proof}
	Suppose $G(\Gamma)$ is a graph product of vertex groups $\{G_v\}$ over a graph $\Gamma$ and for each vertex group $G_v$ we have a completely bounded function $\f_v$ with $\cb{\f_v}< 1+\varepsilon$ and $\f_v(e)=1$.
	Associated with $\f_v$ is a Hilbert space $\scr{H}_v$ and maps $\alpha_v$, $\beta_v :G_v\to \scr{H}_v$ such that $\f_v(h^{-1}g)=\langle\alpha_v(g),\beta_v(h)\rangle$.
	Furthermore, $\|\alpha_v\|$ and $\|\beta_v\|$ are both less than $\sqrt{1+\varepsilon}$ and  $\langle\alpha_v(g),\beta_v(g)\rangle = 1$ for all $g \in G_v$.
	We shall again drop subscripts when not needed.
	As in Proposition \ref{cbext}, for each $g$ and $h \in G_v$, define the map and renormalization constant 
	$$\omega(g):=\tfrac{\alpha(g)+\beta(g)}{2+\sqrt{2\varepsilon}}\text{ and }D(g) := \sqrt{\tfrac{1-\|\omega(g)\|^2}{2}},$$
	and the constants 
	\begin{equation*}
		\begin{split}
			C^\alpha(g,h) := \tfrac{\left\langle\alpha(g),\beta(h)\right\rangle-\left\langle\alpha(g),\omega(h)\right\rangle}{D(h)}
			\text{ and } C^\beta(g,h) := \tfrac{\langle\alpha(h),\beta(g)\rangle - \langle \omega(h),\beta(g)\rangle}{D(h)}.
		\end{split}
	\end{equation*}

	We now have all the data we need to construct our extension except for the index set of our tensor product.
	We shall use $T:=\sqcup_{v\in \Gamma} G(\Gamma)/G(\st(v))$, the left cosets of $G(\st(v))$ for all $v$.
	When $G(\Gamma)$ is a free product, this is identical to the $T$ defined in previous section.
	For each $kG(\st(v))$ in $T$,  define $\scr{H}_{kG(\st(v))}:=\scr{H}_v$ and use $\ambiente{v}$ as its vacuum vector.
	Our Hilbert space is $\scr{\widehat H}  = \bigotimes_{t\in T} (\scr{H}_t \oplus \C^2\oplus \C^2)$.

	Define for $g\equiv g_1\dots g_n \in G(\Gamma)$, 
	\begin{equation*}
		\begin{split}
			\\\widehat\alpha(g) &:= \bigotimes_{0<i\le n} \left\{
			\begin{array}{ll}
				\calpha{g_i}{}_{g_1\dots g_{i-1} G(\st(i))}& \text{ if $g_i$ is in the $d$-tail of $g$}\\
				\ambient{g_i}{}_{g_1\dots g_{i-1} G(\st(i))}& \text{ otherwise,}
			\end{array}\right.
		\end{split}
	\end{equation*}
	and
	\begin{equation*}
		\widehat\beta(g) := \bigotimes_{0<i\le n} \left\{
		\begin{array}{ll}
			\cbeta{g_i}{}_{g_1\dots g_{i-1} G(\st(i))}& \text{ if $g_i$ is in the $d$-tail of $g$}\\
			\ambient{g_i}{}_{g_1\dots g_{i-1} G(\st(i))}& \text{ otherwise.}
		\end{array}\right.
	\end{equation*}
	These are maps from $G(\Gamma)$ to $\scr{\widehat H}$.
	Uncompressed and rearranged so that the $d$-tail is rightmost in the word, we have
	\begin{multline*}
		\widehat\alpha(g) =
		\ambient{g_1}{}_{G(\st(1))}
		\otimes\cdots
		\otimes \ambient{g_{j-1}}{}_{g_1\dots g_{j-2} G(\st(j-1))}
		\\	\otimes{\calpha{g_{j}}{}}_{g_1\dots g_{j-1} G(\st(j))}
		\otimes \\\cdots
		\otimes\calpha{g_n}{}_{g_1\dots g_{n-1}G(\st(n))}
	\end{multline*}
	where $j> n-d\kappa$ by Lemma \ref{dtailsize}.
	Note that the choice functions in $\widehat\alpha$ and $\widehat\beta$ are well defined, as membership in the $d$-tail is independent of representation.
       	Let us show that the indexing of $\widehat\alpha$ and $\widehat\beta$ is well defined.

	We shall focus on $\widehat\alpha$, as the same arguments hold for $\widehat\beta$.
	Suppose we have $g\in G(\Gamma)$ and two equivalent reduced words of $g$ which differ by a single syllable shuffle, say 
	$$g\equiv g_1\dots g_{i}g_{i+1}\dots g_n \equiv g_1\dots g_{i+1}g_{i}\dots g_n$$
	where $G(i)$ and $G(i+1)$ commute.
	Then we have two equalities: 
	\begin{equation*}
		\begin{split}
			g_1\dots g_{i-1} g_i G(\st(i+1)) &= g_1\dots g_{i-1}G(\st(i+1)) \text{ and }\\
			g_1\dots g_{i-1}g_{i+1}G(\st(i)) &= g_1\dots g_{i-1} G(\st(i)),
		\end{split}
	\end{equation*}
	as $g_i$ is in $G(\st(i+1))$ and $g_{i+1}$ is in $G(\st(i))$.
	This implies that the index cosets corresponding to $g_i$ in $\widehat\alpha(g_1\dots g_{i}g_{i+1}\dots g_n)$ and $\widehat\alpha(g_1\dots g_{i+1}g_{i}\dots g_n)$ are the same, and likewise for the index coset of $g_{i+1}$. 
	For any other syllable $g_j$ of $g$, the coset $g_1\dots g_{j-1} G(\st(j))$ is not effected by shuffling $g_i$ and $g_{i+1}$, hence $\widehat\alpha(g_1\dots g_{i}g_{i+1}\dots g_n)=\widehat\alpha(g_1\dots g_{i+1}g_{i}\dots g_n)$.
	Any other reduced word representing $g$ may be reached by a finite sequence of syllable shuffles, so we may unambiguously define $\widehat\alpha(g)$ (similarly $\widehat\beta(g)$ is well defined). 

	Let us show that $\widehat\alpha(g)$ and $\widehat\beta(g)$ are bounded.
	First, $C^\alpha$ and $C^\beta$ are bounded by $8\varepsilon^{1/4}$ by Lemma \ref{approx}.
	Since the only non-unit vectors of $\widehat\alpha(g)$ correspond to syllables in the $d$-tail of $g$, of which there are at most $d\kappa$ by Lemma \ref{dtailsize}, we have
	\begin{equation*}
		\begin{split}
			\|\widehat\alpha(g)\| 
			&= \prod_{\{g_i\text{ in the $d$-tail of $g$}\}}  \|\calpha{g_i}{}\| \\
			&= \prod_{\{ g_i\text{ in the $d$-tail of $g$}\}} \sqrt{ \|\alpha_i(g_i)\|^2+|C^\alpha(g_i,g_i)|^2+|C^\alpha(g_i,e)|^2}
			\\&\le \left(1+\varepsilon+128\sqrt{\varepsilon}\right)^{\frac{d\kappa}{2}}
			\le (1+129\sqrt{\varepsilon})^{\frac{d\kappa}{2}}
		\end{split}
	\end{equation*}
	when $\varepsilon<1$.
	We have the same inequality for $\|\widehat\beta(g)\|$.

	The kernel   
	$$\phi_d(g,h) := \langle\widehat\alpha(g),\widehat\beta(h)\rangle$$
	 is then completely bounded with cb-norm $\cb{\phi_d}=\sup_{g,h\in G}\|\widehat\alpha(g)\|\|\widehat\beta(h)\|\le \left(1+129\sqrt{\varepsilon}\right)^{d\kappa}.$

	For a word $g \equiv g_1\dots g_d$ of length $d$ in $G(\Gamma)$, we have 
	\begin{equation*}
		\begin{split}
			\phi_{d}(g,e) &=
			\prod_{i=1}^d \left \langle \calpha{g_i}{},\ambiente{}\right \rangle_{g_1\dots g_{i-1}G(\st(i))}
			\\ &=\prod_{i=1}^d \langle \alpha(g_i),\beta(e)\rangle
			= \f(g_1)\f(g_2)\dots\f(g_d).
		\end{split}
	\end{equation*}
	where the second equality follows from calculation and the definition of $C^\alpha(g_i,g_i)$.
\end{proof}

For the rest of the section, $\phi_d$ refers to the kernel constructed in Proposition \ref{gpextension}.
As in the free product case, we designed $\phi_d$ to be the kernel of a function restricted to words of length $d$: 
\begin{lemma}\label{invariance}
	For any pair $g,h\in G(\Gamma)$ such that $|h^{-1}g|_r=d$, we have that $\phi_d(g,h)=\phi_d(h^{-1}g,e)$.
\end{lemma}
\begin{proof}
	Let us show that $\phi_{d}(g,h) = \phi_{d}(h^{-1}g,e)$ whenever $|h^{-1}g|_r =d$.

	Suppose $g$ and $h$ are elements of $G(\Gamma)$, and let $n=|g|_r$, $m=|h|_r$, and $d= |h^{-1}g|_r$.
	By Lemma \ref{hg}, there are reduced representations $g\equiv g_1\dots g_n$ and $h\equiv h_1\dots h_m$ such that $g_1\dots g_q= h_1\dots h_q$ and
	$$h^{-1}g\equiv h^{-1}_m\dots h^{-1}_{q+p+1} (h^{-1}_{q+1} g_{q+1})\dots(h^{-1}_{q+p}g_{q+p})g_{q+p+1}\dots g_n$$ for some $q\le\min(m,n)$ and $p\le\kappa$. 

	We showed in the proposition above that 
	\begin{equation*}
		\phi_d(h^{-1}g,e) = 
		\prod_{i=q+p+1}^m\f(h^{-1}_{i})
		\prod_{i=q+1}^{q+p}\f(h_i^{-1}g_i)
		\prod_{i=q+p+1}^n\f(g_{i}).
	\end{equation*} 
	Now we analyze $\phi_d(g,h)=\langle\widehat\alpha(g),\widehat\beta(h)\rangle$.
	Suppose $i\le q$, in which case $g_i=h_i$ and $g_1\dots g_{i-1}G(\st(i))=h_1\dots h_{i-1}G(\st(i))$.
	Then the $g_1\dots g_{i-1}G(\st(i))$  factor of $\langle\widehat\alpha(g),\widehat\beta(h)\rangle$ is one of the following:
	\begin{equation*}
		\begin{split}	
			\big\langle\calpha{g_i}{}&,\cbeta{g_i}{}\big\rangle,
			\\	\big\langle\calpha{g_i}{}&,\ambient{g_i}{}\big\rangle,
			\\\big\langle\ambient{g_i}{}&,\cbeta{g_i}{}\big\rangle, \text{ or } 
			\\\big\langle\ambient{g_i}{}&,\ambient{g_i}{}\big\rangle,
		\end{split}
	\end{equation*}
	depending on whether or not $g_i$ is in the $d$-tails of $g$ or $h$.
	In any case, the factor is equivalent to $1$ by the choice of constants.

	If $q< i \le q+p$, then $g_i\ne h_i$, but they amalgamate in the product $h^{-1}g$.
	By Lemma \ref{cancel}, $g_i$ and $h_i$ are in the $d$-tails of their respective words.
	Since $g_{q+j}$ and $h_{q+j}$ are in $G(\st(i))$ for all $1\le j< i$, we have  
	$$g_1\dots g_q g_{q+1}\dots g_{i-1}G(\st(i)) = g_1\dots g_qG(\st(i)) = h_1\dots h_qh_{q+1}\dots h_{i-1}G(\st(i))$$
	which implies that the $g_1\dots g_qG(\st(i))$ factor of the inner product is 			$$\big\langle\calpha{g_i}{},\cbeta{h_i}{}\big\rangle=\langle\alpha(g_i),\beta(h_i)\rangle.$$

	For all other non-vacuum indices we either have a factor of 
	\begin{equation*}
		\begin{split}	
			&\left\langle\calpha{g_i}{},\ambiente{}\right\rangle 
			= \langle\alpha(g_i),\beta(e)\rangle, \text{ or }
			\\&\left\langle\ambiente{},\cbeta{h_i}{}\right\rangle
			= \langle\alpha(e),\beta(h_i)\rangle. 
		\end{split}
	\end{equation*}

	Hence 
	\begin{equation*}
		\begin{split}
			\phi_{d}(g,h)
			&=\langle\widehat\alpha(g),\widehat\beta(h)\rangle
			\\&= 	\prod_{i=q+1}^{q+p}\langle\alpha(g_i),\beta(h_i)\rangle 
			\prod_{i=q+p+1}^n\langle\alpha(g_{i}),\beta(e)\rangle 
			\prod_{i=q+p+1}^m\langle\alpha(e),\beta(h_{i})\rangle
			\\&= \prod_{i=q+1}^{q+p}\langle\alpha(h_i^{-1}g_i),\beta(e)\rangle 
			\prod_{i=q+p+1}^n\langle\alpha(g_{i}),\beta(e)\rangle 
			\prod_{i=q+p+1}^m\langle\alpha(h^{-1}_{i}),\beta(e)\rangle
			\\&=\prod_{i=q+1}^{q+p}\f(h_i^{-1}g_i)
			\prod_{i=q+p+1}^n\f(g_{i})
			\prod_{i=q+p+1}^m\f(h^{-1}_{i})
			= \phi_{d}(h^{-1}g,e)
		\end{split}
	\end{equation*}
	and we have our desired equality.
\end{proof}

For a natural number $n$ and real number $x$, let $q(n,x):= 2\sqrt{n\kappa 129}\left(1+129\sqrt{x}\right)^{n\kappa}x^{1/4}$.
Clearly for a fixed $n$, the function $q(n,x)$ goes to $0$ as $x$ goes to $0$.

\begin{lemma}\label{cbpd}
	Given the hypotheses of Proposition \ref{gpextension}, there is a positive definite kernel $\Psi$ on $G(\Gamma)$ of norm $1$ such that $\cb{\phi_d-\Psi}  \le q(d,\varepsilon)$ for all natural numbers $d$.
\end{lemma}

\begin{proof}
	Let  $\omega$, $D$, and $\scr{\widehat H}$ be defined as in Proposition \ref{gpextension}.

	Define, for $g\equiv g_1\dots g_n$,	
$$\widehat\omega(g) := \bigotimes_{0<i\le n}\ambient{g_i}{}_{g_1\dots g_{i-1}G(\st(i))},$$ 
By the analysis of $\widehat\alpha$ in Prop \ref{gpextension}, the indexing of $\widehat\omega$ is well defined, and hence $\widehat\omega$ is a well defined map from $G(\Gamma)$ to $\scr{\widehat H}$.
Also, $\|\widehat\omega\|=1$, since $\|\ambient{g}{}\|=1$ for all $g\in G(\Gamma)$ by the definition of $D(g)$.

Our positive definite kernel is then $\Psi(g,h):=\langle\widehat\omega(g),\widehat\omega(h)\rangle $.
Note that this is not $G(\Gamma)$-invariant.

Now let $d$ be a natural number and let us show that $\cb{\phi_d-\Psi}  \le q(d,\varepsilon)$.
	Let $g\equiv g_1\dots g_n$ and $h$ be elements of $G$.
	Recall that $\phi_d(g,h) = \langle\widehat\alpha(g),\widehat\beta(h)\rangle$, where the maps $\widehat\alpha(g)$ and $\widehat\beta(g)$ differ from $\widehat\omega(g)$ in value only on the $d$-tail of $g$.
	Since the shared values are unit vectors, we have that
	\begin{equation*}
		\begin{split}
			\langle\widehat\alpha(g),\widehat\omega(g)\rangle 
			&= \prod_{\{g_i\text{ in $d$-tail of $g$}\}}\langle\calpha{g_i}{},\ambient{g_i}{}\rangle
			\\&= \prod_{\{g_i\text{ in $d$-tail of $g$}\}}\langle\alpha(g_i),\beta(g_i)\rangle
			=1
		\end{split}
	\end{equation*}
	and hence
	\begin{equation*}
		\begin{split}
			\|\widehat\alpha(g)-\widehat\omega(g)\|^2
			&= \|\widehat\alpha(g)\|^2 + \|\widehat\omega(g)\|^2-2\langle\widehat\alpha(g),\widehat\omega(g)\rangle
			=\|\widehat\alpha(g)\|^2-1.
		\end{split}
	\end{equation*}
	The same analysis shows that $\|\widehat\beta(g)-\widehat\omega(g)\|^2 \le \|\widehat\beta(g)\|^2-1.$

	As for the kernels $\phi_d$ and $\Psi$ themselves, we can decompose their difference into the sum of two completely bounded functions:
	\begin{equation*}
		\begin{split}
			\phi_d(g,h) - \Psi(g,h) 
			&= \langle\widehat\alpha(g),\widehat\beta(h)\rangle -  \langle\widehat\omega(g),\widehat\omega(h)\rangle
			\\	&= \langle\widehat\alpha(g)-\widehat\omega(g),\widehat\beta(h)\rangle +  \langle\widehat\omega(g),\widehat\beta(h)-\widehat\omega(h)\rangle.
		\end{split}
	\end{equation*}

	Hence 
	\begin{equation*}
		\begin{split}
			\cb{\phi_d- \Psi}
			&\le \cb{\langle\widehat\alpha(\cdot)-\widehat\omega(\cdot),\widehat\beta(\cdot)\rangle} + \cb{ \langle\widehat\omega(\cdot),\widehat\beta(\cdot)-\widehat\omega(\cdot)\rangle}
			\\&\le \sup_{g,h\in G}\|\widehat\alpha(g)-\widehat\omega(g)\|\|\widehat\beta(h)\|
			+ \sup_{g,h\in G}\|\widehat\omega(g)\|\|\widehat\beta(h)-\widehat\omega(h)\|
			\\&\le \sup_{g,h\in G}\left(\|\widehat\alpha(g)\|^2-1\right)^{\frac{1}{2}}\|\widehat\beta(h)\|
			+ \sup_{h\in G}\left(\|\widehat\beta(h)\|^2-1\right)^{\frac{1}{2}}
			\\&\le \left((1+129\sqrt{\varepsilon})^{d\kappa}-1\right)^{\frac{1}{2}} (1+129\sqrt{\varepsilon})^{\frac{d\kappa}{2}}+ \left((1+129\sqrt{\varepsilon})^{d\kappa}-1\right)^{\frac{1}{2}}
			\\&\le \left(d\kappa 129(1+129\sqrt{\varepsilon})^{d\kappa}\sqrt{\varepsilon}\right)^{\frac{1}{2}} (1+129\sqrt{\varepsilon})^{\frac{d\kappa}{2}}+ \left(d\kappa 129(1+129\sqrt{\varepsilon})^{d\kappa}\sqrt{\varepsilon}\right)^{\frac{1}{2}}
			\\&\le 2\sqrt{d\kappa 129}(1+129\sqrt{\varepsilon})^{d\kappa}\varepsilon^{\frac{1}{4}}
			= q(d,\varepsilon).
		\end{split}
	\end{equation*}
	which proves the lemma.
\end{proof}

Now let $\Xi_d$ be the characteristic function of the $d$ sphere in the CAT(0)-cube complex $X$ from Section 3 and $\chi_d$ be the characteristic function of words of reduced length $d$ in $G(\Gamma)$.
For a pair $g,h\in G(\Gamma)$ (and a base point $x_0\in X$), we have
$$\chi_d(h^{-1}g)= \Xi_{2d}(gx_0,hx_0) = \left\{
\begin{array}{lr}
	1 &\text{if $|h^{-1}g|_r = d$,}\\
	0 &\text{else.}
\end{array}\right.$$
We quote following theorem of Mizuta:
\begin{theorem}[(Theorem 2 from \cite{Mizu08})]
	Let $X$ be a finite-dimensional CAT(0) cube complex and let $\chi_n = \{(x, y) \in X\times X: d_X(x, y) = n\}$ for $n \in \Z^+$. Then the norms of Schur multipliers of the characteristic function  $\chi_n$ increase polynomially, i.e. there exists a polynomial $p$ such that $\cb{\chi_n}\le  p(n)$.	
\end{theorem}
Hence $\cb{\chi_d}\le p(d)$ for some polynomial $p$.
Given a fixed $r>0$ and finite sum of the form $\sum_{d=0}^M e^{-d/r}\chi_d$,  we note that 
$$\cb{\sum_{d=0}^M e^{-d/r}\chi_d} 
\le \sum_{d=0}^M e^{-d/r}\cb{\chi_d} 
\le \sum_{d=0}^M e^{-d/r}p(d), $$
which converges as $M$ goes to infinity.
The limit function $\sum_{d=0}^\infty e^{-d/r}\chi_d = e^{|\cdot|_r/r}$ is positive definite.

We now have everything necessary to prove the main theorem.

\begin{theorem}\label{main}
	The graph product of weakly amenable groups (with Cowling-Haagerup constant 1) is weakly amenable.
\end{theorem}

\begin{proof}
	Suppose we are given a graph product $G(\Gamma)$ with weakly amenable vertex groups $\{G_v\}$,
	 a small $\varepsilon>0$, and a finite subset $S$ of $G(\Gamma)$. 
	 We show that there is a completely bounded function $\mathrm{F}$ such that $|1-\mathrm{F}(g)|\le \varepsilon$ for all $g\in S$ and $\cb{\mathrm{F}}\le 1+\varepsilon$.

	First we need to define some constants.
	Let $N = \max\{|g|_r :\; g \in S\}$.
	Choose $\delta$ such that, if you have $N+1$ numbers $\{a_i\}$ each of which is at most $\delta$ away from $1$, then $|1-\prod_1^{N+1}a_i| < \varepsilon$. 
	Choose $r$ and $M$ large enough such that $|1-e^{-N/r}|<\delta$ and $\sum_{d=M+1}^\infty e^{-d/r}p(d) < \frac{\varepsilon}{2}$.
	Let $\mu>0$ be small enough that  $q(M,\mu)\sum^M_{d=0}e^{-d/r}p(d)
	\le \frac{\varepsilon}{2}$.

	For each $v$, let $S_v := \{g_v \in G_v\; |\; g_v \text{ is a syllable of some }g \in S\}$ and,
	by the weak amenability of $G_v$, choose a finitely supported completely bounded function $\f_v$ on $G_v$ such that $\cb{\f_v}\le 1+\mu $ and $|1 - \f_v(g_v)| < \delta$ for all $g_v\in S_v$. 
	Using Proposition \ref{gpextension}, for each $d\in\N$, we can construct the completely bounded kernel $\phi_d$ on $G(\Gamma)$,
	and by Lemma \ref{cbpd} there exists an ambient positive definite kernel $\Psi$ on $G(\Gamma)$ such that $\cb{\phi_d-\Psi}<q(d,\mu)$.

	Our final function is then $\bF = \sum_{d=1}^M e^{-d/r}\phi_d\chi_d$.
	For any $g\equiv g_1\dots g_n \in S$, we have that 
	\begin{equation*}
		\begin{split}
			\mathrm{F}(g) = e^{-|g|_r/r}\phi_n(g,e) = e^{-|g|_r/r}\f_1(g_1)\f_2(g_2)\dots \f_n(g_n)
		\end{split}
	\end{equation*}
	and each of the at most $N+1$ terms in this product is within $\delta$ of $1$.
	Hence, by our definition of $\delta$,  $|1-\mathrm{F}(g)| < \varepsilon$.
	More generally, $F$ is non-zero only on words with at most $M$ syllables.
	On such a word, it is a product of $\f$ applied to each syllable. 
	As each $\f$ is finitely supported, $F$ is only non-zero on finitely many words.
	We also have that 
	\begin{equation*}
		\begin{split}
			\cb{\mathrm{F}} 
			&=\cb{\sum^M_{d=0}e^{-d/r}\phi_d\chi_d}
			\le \cb{\sum^M_{d=0}e^{-d/r}(\phi_d-\Psi)\chi_d}+\cb{\sum^M_{d=0}e^{-d/r}\Psi\chi_d}
			\\&\le \sum^M_{d=0}e^{-d/r}\cb{\phi_d-\Psi}\cb{\chi_d}+\sum^\infty_{d=M+1}e^{-d/r}\cb{\Psi}\cb{\chi_d} + \cb{\sum^\infty_{d=0}e^{-d/r}\Psi\chi_d}
			\\&\le \sum^M_{d=0}e^{-d/r}q(d,\mu)p(d)+\sum^\infty_{d=M+1}e^{-d/r}p(d) + \cb{e^{-|\cdot|_r/r}\Psi}
			\\&\le q(M,\mu)\sum^M_{d=0}e^{-d/r}p(d)
			+\frac{\varepsilon}{2} + 1 
			\\&\le 1 + \varepsilon.
		\end{split}
	\end{equation*}
	The third line follows from Lemma \ref{cbpd} and the fact that $\cb{\Psi}=1$.

	Given a net of finite subsets that exhausts $G(\Gamma)$ and numbers $\varepsilon_n$ which go to $0$,  we can construct a net of finitely supported completely bounded functions which approximate the identity of $G(\Gamma)$. 
	Further, the cb-norms go to one, hence $G(\Gamma)$ is weakly amenable with Cowling-Haagerup constant one.
\end{proof}

\bibliography{ear}{}
\bibliographystyle{plain}
\address{   Eric Reckwerdt\\
 	University of \Hawaii, \Manoa, \\
	Department of Mathematics, \\
	2565 McCarthy Mall, \\
	Honolulu, HI 96822-2273\\
      	USA	\\
      \email{ear@hawaii.edu}}

\end{document}